\documentclass[10pt,a4paper]{amsart}
\usepackage{amsmath,amssymb,amsthm,mathtools}
\usepackage{enumitem,booktabs,tabularx,longtable}
\usepackage{hyperref, color}

\theoremstyle{plain}
\newtheorem{theorem}{\bf Theorem}[section]
\newtheorem{proposition}[theorem]{\bf Proposition}
\newtheorem{lemma}[theorem]{\bf Lemma}

\theoremstyle{definition}

\newcommand{\N}{\mathbb N}
\newcommand{\Z}{\mathbb Z}
\newcommand{\R}{\mathbb R}

\newcommand{\F}{\mathbb F}
\newcommand{\C}{\mathbb C}

\newcommand{\dd}{\mathsf d}

\DeclareMathOperator{\ord}{ord}
\DeclareMathOperator{\supp}{supp}

\DeclareMathOperator{\Hom}{Hom}

\numberwithin{equation}{section}

\date{}

\title[On a classical zero-sum invariant II]{On a classical zero-sum invariant II: \\ Disproof of a long-standing conjecture}

\author{Alfred Geroldinger \and Guoqing Wang \and Wenkai Yang}

\address{Department of Mathematics and Scientific Computing\\ University of Graz, NAWI Graz\\ Heinrichstra{\ss}e 36\\ 8010 Graz, Austria}
\email{alfred.geroldinger@uni-graz.at}
\urladdr{https://geroldinger.github.io/}

\address{School of Mathematical Sciences \\ Tiangong University \\ Tianjin 300387, P.R. China}
\email{gqwang1979@aliyun.com}

\address{Center for Combinatorics \\ Nankai University, Tianjin 300071, P.R. China}
\email{ywk1@tju.edu.cn}
\email{9820250116@nankai.edu.cn}


\begin{document}

\begin{abstract}
For a nontrivial finite abelian group $G$, let $\nu(G)$ be the smallest {nonnegative} integer
$\ell$ such that every zero-sum free sequence $T$ over $G$ of length at least $\ell$ has the following property: all nonzero elements of $G$ that do not occur as a subsequence sum of $T$ lie in a proper coset of some subgroup of $G$. It is easy to check that $\mathsf d (G)-1 \le \nu (G) \le \mathsf d (G)$, where $\mathsf d (G)$ is the small Davenport constant of $G$. A conjecture by Gao from the year 2000 stated that equality should always hold at the lower bound. This conjecture has since been confirmed for many families of groups (including all p-groups and groups of rank at most two). In this paper, we disprove the conjecture.
\end{abstract}

\maketitle

\section{Introduction}

Let $G$ be a nontrivial finite abelian group, say $G = C_{n_1} \oplus \ldots \oplus C_{n_r}$, where $r=\mathsf r (G) \in \mathbb N$ is the rank of $G$ and $1 < n_1 \mid \ldots \mid n_r$. We set
$G^\bullet=G\setminus\{0\}$, $\mathsf d^* (G) = \sum_{i=1}^r (n_i-1)$, and  $\mathsf D^* (G) = 1 + \mathsf d^* (G)$. The (small) Davenport constant $\mathsf d (G)$ is the maximal length of a zero-sum free sequence over $G$,  the (large) Davenport constant $\mathsf D (G)$ is the maximal length of a minimal zero-sum sequence over $G$, and we have $\mathsf D (G) = \mathsf d (G)+1 \ge \mathsf D^* (G)$. {The equality $\mathsf D(G)=\mathsf D^*(G)$ holds for groups of rank at most two and for $p$-groups.} In general, strict inequality may hold.
If $r \le 2$, then the structure of minimal zero-sum sequences
of maximal length is known.
For $r \ge 3$, little is known about the structure of such sequences even for $p$-groups.

The invariant $\nu (G)$ was introduced by Emde Boas in the 1960s (\cite[page 15]{Em69a}, \cite{Em-Kr69}), and we call it the {\it Emde-Boas constant}. It offers information on the structure of extremal sequences by considering the set of their subsequence sums. Let $p$ be a prime divisor of $|G|$. Let $\nu (G)$ resp. $\nu_p (G)$ be the smallest {nonnegative} integer such that, for any zero-sum free sequence $T$ over $G$, $|T| \ge \nu (G)$ resp. $|T| \ge \nu_p (G)$ implies that
\begin{equation} \label{def-nu(G)}
G^{\bullet} \setminus \Sigma (T) \subseteq \alpha+H \ \mbox {for some  subgroup $H  \subsetneq G$ and some $\alpha\in G\setminus H$}
\end{equation}
resp.
\begin{equation} \label{def-nu_p(G)}
G^{\bullet} \setminus \Sigma (T) \subseteq \alpha + H \ \text{for some subgroup} \ H \subseteq G \ \text{with} \ (G \colon H) = p \ \text{and some} \ \alpha \in G \setminus H \,.
\end{equation}
The refinement $\nu_p (G)$ was introduced recently in \cite{Ge-Ya27a}{. If} $T \in \mathcal F(G)$ is zero-sum free with $|T|=\mathsf d(G)$, then a straightforward argument gives $\Sigma(T)=G^{\bullet}$,
whence $\nu(G)\le\nu_p(G)\le\mathsf d(G)$ by definition. {A further simple argument (see, for example, \cite[Proposition 3.2.5]{Ge-Gr-Zh26a}) shows that}
\begin{equation} \label{basic-inequ-2}
\mathsf d (G) -1 \le \nu (G) \le \nu_p (G) \le \mathsf d (G) \,.
\end{equation}

In 2000  (\cite[Section~3]{Ga00b}), Weidong Gao proposed the following conjecture.

\medskip
\noindent
{\bf Conjecture A.} \label{conj:gao}
Every nontrivial finite abelian group $G$ satisfies $\nu(G)=\dd(G)-1$.

\smallskip
This conjecture initiated a great deal of  further research (e.g.,\cite{Sc11b}). {It was confirmed} for $p$-groups (\cite[Theorem 5.5.9]{Ge-HK06a}) and for  groups of rank at most two (\cite[Theorem 3.4.11]{Ge-Gr-Zh26a}).  A new approach focused on groups of the form $C_2^r \oplus C_{2n}$. {These groups have received considerable attention} in Zero-Sum Theory, as they combine two extremal cases, namely elementary $2$-groups and cyclic groups. The following result was shown only recently (\cite{Ge-Ya27a}).

\newpage
\noindent
{\bf Theorem B.} Let $G = C_2^r \oplus C_{2n}$ with $r, n \in \N$.
\begin{enumerate}
\item If $r \in \{1,2\}$, then $\nu (G) = \nu_2 (G) = \mathsf d^* (G)-1= \mathsf d (G)-1$.

\item If $r=4$ and $n>70$ odd, then  $\nu (G) = \nu_2 (G) = \mathsf d^* (G) = \mathsf d (G)-1$.
\end{enumerate}

\smallskip
In the present paper, {we not only disprove Conjecture A but also demonstrate} the erratic behavior of $\nu (G)$ even within  groups of the form $C_2^r \oplus C_{2n}$ (compare Theorem B, Theorem \ref{thm:main1}, and Theorem \ref{thm:main2}). In particular, both cases $\nu (G) = \mathsf d (G)-1$ and $\nu (G) = \mathsf d (G)$ can occur, regardless of whether $\mathsf d (G) = \mathsf d^* (G)$ or $\mathsf d (G) > \mathsf d^* (G)$. {Outside $p$-groups and groups of rank two, the precise value of $\mathsf d(G)$ is known in only a limited number of cases. Several results give groups with $\mathsf d(G)>\mathsf d^*(G)$ (\cite{Ma92,Ge-Li-Ph12,Li20a,Zh23a,Wa28a}), but its exact value is known in only a few such cases (\cite{Sa-Ch14a,Zh23a}).}

The following two theorems provide the first examples of groups $G$ with $\nu (G)=\mathsf d (G)$.

\begin{theorem}\label{thm:main1}
Let $G=C_2^3\oplus C_{2n}$ with $n \in \N$. Then
\[
 \nu(G)=\nu_2(G)=
 \begin{cases}
 \dd^* (G) = \dd(G),    &n\ge3\ \text{odd},\\
 \dd^* (G)-1 = \dd(G)-1,  &n=1\ \text{or }n\ \text{even}.
 \end{cases}
\]
\end{theorem}

\begin{theorem}\label{thm:main2}
Let $G=C_2^r\oplus C_{2n}$ with $r\in\{5,6\}$ and $n \in \N$.
\begin{enumerate}
\item If $n=3$, then $\nu(G)=\nu_2(G)=\dd(G)=\dd^*(G)+1$.

\item If $n\in\{1,2\}$ or $4\mid n$, then $\nu(G)=\nu_2(G)=\dd^* (G)-1=\dd(G)-1$.
\end{enumerate}
\end{theorem}

\smallskip
{To prove these results, we develop} new techniques for studying $\nu (G)$ (see Section \ref{sec:3} and Theorem \ref{thm:large}), and we also show that the phenomenon $\nu (G)=\mathsf d (G)$ is not an exceptional one (see Proposition \ref{prop:group-extension}).  {The proofs of Theorems~\ref{thm:main1} and \ref{thm:main2} are given in Sections~\ref{sec:4} and \ref{sec:5}, respectively.} {We conclude with the following open question.}

\medskip
\noindent

{\bf Open Problem.} Let $G$ be a {nontrivial} finite abelian group. Does  $\nu (G) = \mathsf d (G)-1$ hold if $G$ has rank three or if $G = C_n^r$ with $r,n \in \N$? {It is not known whether} $\mathsf d (G)=\mathsf d^* (G)$  holds for these classes of groups.

\section{Background on Zero-Sums and on the Davenport Constant} \label{sec:background}

We will briefly present key concepts and notation used throughout this paper. We denote by $\N$ the set of positive integers and set $\N_0 =\N \cup \{0\}$. For $x, y \in \R$, we use $[x,y] = \{ m \in \Z \colon x \le m \le y\}$ for the discrete interval between $x$ and $y$. For $n \in \N$, let $C_n$ denote a cyclic group of order $n$.

Let $G$ be a finite abelian group. Then $G[2]=\{g\in G:2g=0\}$ is an elementary $2$-group, and elementary $2$-groups will be considered as $\F_2$-vector spaces whenever convenient. A tuple $(e_1, \ldots, e_r)$ is a {\it basis} of $G$ if $e_i \ne 0$ for all $i \in [1,r]$ and $\langle \{e_1, \ldots, e_r\} \rangle = \langle e_1 \rangle \oplus \ldots \oplus \langle e_r \rangle = G$.

If $G_0 \subseteq G$ is a subset, then $\langle G_0 \rangle$ is the subgroup generated by $G_0$.
We denote by $\mathcal F (G_0)$ the free abelian monoid with basis $G_0$, whose elements are called {\it sequences} over $G_0$. If $S = g_1 \cdot \ldots \cdot g_{\ell} = \prod_{g \in {G_0}} g^{\mathsf v_g (S)} \in \mathcal F (G_0)$, then
\[
\begin{aligned}
|S| & = \ell  \in \N_0 \quad \text{the length of $S$, } \\
\mathsf h (S) & = \max \{ \mathsf v_g (S) \colon g \in G\} \in \N_0 \quad \text{the maximum multiplicity of a term of $S$, } \\
\sigma (S) & = g_1 + \ldots + g_{\ell} \quad \text{the sum of $S$,  } \\
\Sigma (S) & = \Big\{ \sum_{i \in I} g_i \colon \emptyset \ne I \subseteq [1, \ell] \Big\} \quad \text{the set of subsequence sums of $S$,}  \\
\Sigma_k (S) & = \Big\{ \sigma (T) \colon T \mid S, |T|=k \big\} \quad \text{for every $k \in \N$}, \quad \text{and} \\
\Sigma^*(S) & = \Sigma (S) \cup \{0\}.
\end{aligned}
\]
If $G'$ is a group and $\varphi \colon G \to G'$ is a map, then $\varphi (S) = \varphi (g_1) \cdot \ldots \cdot \varphi (g_{\ell})$, and $-S=(-g_1) \cdot \ldots \cdot (-g_{\ell})$. We say that $S$ is {\it zero-sum free} if $0 \notin \Sigma (S)$, a {\it zero-sum sequence} if $\sigma (S)=0$, and a {\it minimal zero-sum sequence} {if it is a nontrivial zero-sum sequence without any proper nontrivial zero-sum subsequence.}
Let $\eta (G)$ denote the smallest integer $\ell \in \N$ such that every sequence $S \in \mathcal F (G)$ of length $|S| \ge \ell$ has a zero-sum subsequence $T$ of length $|T| \in [1, \exp (G)]$.

We recall some basic properties of the Davenport constant $\mathsf d (G)$ and of $\eta (G)$, which will be used in the sequel without further mention.
Let $G = C_{n_1} \oplus \ldots \oplus C_{n_r}$ with $1 < n_1 \mid \ldots \mid n_r$. Then
\begin{equation} \label{basic-inequ-1}
\sum_{i=1}^r (n_i-1) = \mathsf d^* (G) \le \mathsf d (G) \,,
\end{equation}
with equality for $p$-groups. If ${r=2}$, then (\cite[Theorem 5.8.3]{Ge-HK06a})
\begin{equation} \label{basic-inequ-3}
\mathsf d (G) = \mathsf d^* (G) = n_1+n_2-2 \quad \text{and} \quad \eta (G) = 2n_1+n_2-2 \,.
\end{equation}
{For a cyclic group $G=C_n$, we have $\eta(G)=\mathsf d(G)+1=n$.}
If $S \in \mathcal F (G)$ is zero-sum free, then $G^{\bullet} \setminus \Sigma (S) = G \setminus \Sigma^* (S)$ and if $|S|=\mathsf d (G)$, then $\Sigma (S) = G^{\bullet}$. If $G=G_1 \oplus G_2$, then
\begin{equation} \label{basic-inequ-4}
\mathsf d^* (G_1) + \mathsf d^* (G_2) \le \mathsf d (G_1)+\mathsf d (G_2) \le \mathsf d (G) \,.
\end{equation}

{We will also use the following results on the Davenport constant and the eta constant.}
\begin{lemma}\label{2.1}
Let $G=C_2^r\oplus C_{2n}$ with $r,n \in \N$.
\begin{enumerate}
\item If $r\le2^{\mathsf v_2(n)+1}-1$, then $\dd(G)=\dd^*(G)$, where $\mathsf v_2(n)$ is the largest
      integer $h\ge0$ such that $2^h\mid n$. In particular, this equality holds for even $n$ and $r\le3$.

\item If $n\ge3$ is odd, then $\dd(G)=\dd^*(G)$ if and only if $r\le3$.
\end{enumerate}
\end{lemma}
\begin{proof}
See \cite[Theorems~3.9.2 and 3.9.7]{Ge-Gr-Zh26a}.
\end{proof}

\begin{lemma} \label{2.2}
Let $G_0$ be a  finite abelian  $p$-group and let $m$ be a power of $p$ such that  $\mathsf d (G_0) \le m-1$. Then
\[
\mathsf d (G_0 \oplus C_{km}) = \mathsf d^* (G_0 \oplus C_{km}) \quad \text{for all $k \in \N$} \,.
\]
\end{lemma}

\begin{proof}
See \cite[Theorem 3.3.10]{Ge-Gr-Zh26a}.
\end{proof}

\begin{lemma} \label{2.3}
Let $G$ be a  finite abelian  $p$-group with $\mathsf D (G) \le 2 \exp (G)-1$. Then $\eta (G) = 2 \mathsf D (G)-\exp (G)$.
\end{lemma}

\begin{proof}
See \cite[Theorem 1.6]{Luo}.
\end{proof}

\begin{lemma} \label{orth-relations}
Let $G$ be a finite abelian group, let $\widehat G =  \Hom (G, \C^{\times})$ be its character group, and let $\chi_0$ be the trivial character.
\begin{enumerate}
\item For any $\chi\in \widehat G$,
            $$\sum_{g \in G} \chi (g) =
            \begin{cases}
            |G| &\mbox{ if} \quad \chi = \chi_0, \\
            0   &\mbox{ if} \quad \chi \ne \chi_0.
            \end{cases}$$
      For any $g\in G$,
            $$\sum_{\chi \in \widehat G} \chi (g) =
            \begin{cases}
            |G| &\mbox{ if} \quad g=0, \\
            0   &\mbox{ if} \quad g \ne 0.
             \end{cases}$$

\item \emph{(Orthogonality Relations)} For any $\chi,\,\psi\in\widehat G$,
     $$\sum_{g \in  G} \chi (g) \psi^{-1} (g) =
      \begin{cases}
      |G| &\text{ if} \quad \chi = \psi, \\
       0\   &\text{ if} \quad \chi \ne  \psi.
      \end{cases}$$
      For any $g,\,h\in G$,
      \[
      \sum_{\chi \in \widehat G} \chi (g) \chi^{-1} (h) =
      \begin{cases}
      |G| &\text{ if} \quad g=h, \\
      0   &\text{ if} \quad g \ne h.
      \end{cases}
      \]
\end{enumerate}
\end{lemma}

\begin{proof}
	See \cite[Proposition 5.5.2]{Ge-HK06a}.
\end{proof}

By \cite{Ma92}, $\mathsf d (C_2^r \oplus C_6) - \mathsf d ^* (C_2^r\oplus C_6)$ {tends to infinity as $r\to\infty$.}

\begin{proposition}\label{lem:davenport}
Let $G = C_2^r \oplus C_{6}$ with $r \in \N$.
Then $\mathsf d (G) = \mathsf d^* (G) +1$ if and only if $r \in [4,6]$.
\end{proposition}

\begin{proof}
By Lemma~\ref{2.1}, it suffices to consider $r \ge 4$.  We start with the following two assertions.

\begin{enumerate}
\item[{\bf A1.}\,]   $\mathsf d (C_2^7 \oplus C_6) \ge \mathsf d^* (C_2^7 \oplus C_6) + 2 = 14$.

\smallskip

\item[{\bf A2.}\,]  $\mathsf d (C_2^6 \oplus C_6) \le \mathsf d^* (C_2^6 \oplus C_6) + 1= 12$.
\end{enumerate}

\smallskip
If {\bf A1} holds and $r \ge 7$, then
\[
\mathsf d (C_2^r \oplus C_6) \ge \mathsf d (C_2^7 \oplus C_6) +  (r-7) \ge \mathsf d^* (C_2^7\oplus C_6) + (r-5) = \mathsf d^* (C_2^r \oplus C_6) + 2 \,.
\]
If {\bf A2} holds, then
\[
\mathsf d^* (C_2^6 \oplus C_6) + 1 \ge \mathsf d (C_2^6 \oplus C_6) \ge \mathsf d (C_2^5 \oplus C_6)+1 \ge \mathsf d (C_2^4 \oplus C_6)+2 \,{.}
\]
{By Lemma~\ref{2.1}.2 and the integrality of $\mathsf d(C_2^r\oplus C_6)$, we also have $\mathsf d(C_2^r\oplus C_6)\ge\mathsf d^*(C_2^r\oplus C_6)+1$ for $r\ge4$. These bounds give $\mathsf d(C_2^r\oplus C_6)=\mathsf d^*(C_2^r\oplus C_6)+1$ for $r\in[4,6]$.} Thus, it remains to prove {\bf A1.} and {\bf A2.}.

\smallskip
\noindent
\begin{proof}[Proof of {\bf A1}]
Let $(e_1,\ldots,e_8)$
		be a basis of $C_2^7 \oplus C_6$ with $\operatorname{ord}(e_i)=2$ for all
		$i\in[1,7]$ and $\operatorname{ord}(e_8)=6$. Let
		\[
		H_1=\langle e_1,e_2,e_3\rangle
		\quad\text{and}\quad
		H_2=\langle e_5,e_6,e_7\rangle.
		\]
We assert that the sequence
\[
S= \prod_{u\in H_1\setminus\{0\}}(e_4+u+e_8)  \prod_{v\in H_2\setminus\{0\}}(v+e_8)
\]
is zero-sum free. This implies that $\mathsf d (C_2^7 \oplus C_6) \ge |S| = 14 = \mathsf d^* (C_2^7 \oplus C_6) + 2$.
Assume to the contrary that $S$ has a nontrivial zero-sum
		subsequence $T$. Then there exist subsets
		$A\subset H_1\setminus\{0\}$ and
		$B\subset H_2\setminus\{0\}$ such that
		\[
		0=\sigma(T)
		=\sum_{u\in A}u+\sum_{v\in B}v
		+|A|e_4+(|A|+|B|)e_8.
		\]
		Since
		${C_2^7\oplus C_6}=H_1\oplus H_2\oplus\langle e_4\rangle
		\oplus\langle e_8\rangle$, it follows that
		\[
		\sum_{u\in A}u=\sum_{v\in B}v=0,
		\qquad
		2\mid |A|,
		\qquad
		6\mid |A|+|B|.
		\]
In particular, both $|A|$ and $|B|$ are even. For  $i\in[1,2]$, the sum of all seven nonzero elements
		of $H_i$ is zero. A two-element subset of $H_i\setminus\{0\}$  cannot have sum zero, since its elements
		are distinct. A $6$-element subset  cannot have sum
		zero, since its sum equals the omitted nonzero element.
		Consequently, $|A|,|B|\in\{0,4\}$ and $|T| \ge 1$ implies that
		$|T|=|A|+|B|\in\{4,8\}$, a contradiction to
		$6\mid |A|+|B|$.  \qedhere \ ({\bf A1})
\end{proof}

\smallskip
\noindent
\begin{proof}[Proof of {\bf A2}]
It suffices to verify that a sequence $S \in \mathcal F (C_2^6 \oplus C_6)$ of length $|S|=13$ cannot be zero-sum free. We choose a basis $(e_1, \ldots, e_8)$  of $G= C_2^7 \oplus C_3 \cong C_2^6 \oplus C_6$ with $\ord (e_i)=2$ for $i \in [1,7]$ and with $\ord (e_8)=3$.

Let $S=g_1 \cdot \ldots \cdot  g_{13} \in \mathcal F (G)$ and
assume to the contrary that $S$ is zero-sum free. Let
	$\varphi:G\to G_0 = \langle e_1,\ldots,e_7\rangle$ be the projection with
	$\ker (\varphi) =\langle e_8\rangle$, and let
	  $g_i=\varphi(g_i)+c_i e_8$ with $c_i\in\{0,1,-1\}$ for all $i\in[1,13]$.
For $\boldsymbol v = (v_1, \ldots, {v_{13}}) \in \{0,1\}^{13}$, we consider a subsequence
$T_{\boldsymbol v} = \prod_{i=1}^{13}g_i^{ v_i}$ of $S$, and we define a homomorphism
\[
\psi: (\Z / 2\Z)^{13} \rightarrow G_0 \,, \quad
(v_1+ 2 \Z, \ldots, v_{13}+2\Z) \ \longmapsto \ \sum_{i=1}^{13} v_i\varphi(g_i) =\varphi(\sigma(T_{\boldsymbol v})) \,.
\]
Then  $H= \ker (\psi)$ is an elementary $2$-group with $|H| \ge2^6$, and we denote by    $\widehat H = \operatorname{Hom}({H},\C^\times)$
	its character group with trivial character $\chi_0$. For simplicity of notation, we identify $\boldsymbol v = (v_1, \ldots, {v_{13}}) \in \{0,1\}^{13}$ with $(v_1+2\Z, \ldots, v_{13}+2\Z) \in (\Z / 2\Z)^{13}$.
If $\boldsymbol{0} \ne \boldsymbol v\in\ker (\psi)$, then  $T_{\boldsymbol v} \ne1$ and the
 zero-sum freeness of $S$ implies that
\[
0 \ne \sigma(T_{\boldsymbol v}) = \sum_{i=1}^{13} \big( v_i \varphi (g_i) + v_i c_i e_8 \big) =
\left(\sum_{i=1}^{13}c_i v_i\right)e_8
	\in\{e_8,-e_8\} \,,
\]
whence $\left(\sum_{i=1}^{13}c_i  v_i\right)^2 \equiv1\mod3$.  {We continue with three assertions.}

\smallskip
\noindent
\begin{enumerate}
\item[{\bf A3.}\,]  $|H|=64$.
\end{enumerate}

\begin{proof}[Proof of {\bf A3}]
Since $H$ is an elementary $2$-group, every character takes
	values in $\{1,-1\}$ and hence equals its complex conjugate.
	For each $i\in[1,13]$, the map
	$\boldsymbol v\mapsto(-1)^{ v_i}$ on $H$ is a character.
	Since $ v_i\equiv(-1)^{ v_i}-1\mod3$, we have
	\begin{equation}\label{eq:davenport-character-expansion}
		\sum_{i=1}^{13}c_i v_i
		\equiv
		c_1(-1)^{ v_1}+\ldots+c_{13}(-1)^{ v_{13}}
		-(c_1+\ldots+c_{13})(-1)^0
		\mod3.
	\end{equation}
	The right-hand side of \eqref{eq:davenport-character-expansion}
	involves at most fourteen characters, including the trivial
	character.
	
	If $\boldsymbol v=\mathbf0$, then
	$\left(\sum_i c_i v_i\right)^2=0$ and, by
	Lemma~\ref{orth-relations},
	$\sum_{\chi\in\widehat H}\chi(\mathbf0)=|H|$.
	If $\boldsymbol v\in H\setminus\{\mathbf0\}$, then the zero-sum
	freeness of $S$ gives
	$\left(\sum_i c_i v_i\right)^2\equiv1\mod3$,
	whereas Lemma~\ref{orth-relations} gives
	$\sum_{\chi\in\widehat H}\chi(\boldsymbol v)=0$.
	Thus, for every $\boldsymbol v\in H$,
	\begin{equation}\label{eq:davenport-square-values}
		|H|\left(\sum_i c_i v_i\right)^2
		\equiv |H|-\sum_{\theta\in\widehat H}\theta(\boldsymbol v)
		\mod3.
	\end{equation}
	
Let $\chi\in\widehat H\setminus\{\chi_0\}$.
	Lemma~\ref{orth-relations} gives
	$\sum_{\boldsymbol v\in H}\chi(\boldsymbol v)=0$ and
	$\sum_{\theta\in\widehat H}\sum_{\boldsymbol v\in H}
	\theta(\boldsymbol v)\chi(\boldsymbol v)=|H|$.
	{When we multiply \eqref{eq:davenport-square-values} by $\chi(\boldsymbol v)$ and sum over $\boldsymbol v\in H$, the term $|H|\sum_{\boldsymbol v\in H}\chi(\boldsymbol v)$ vanishes.}
	Hence
\[
		|H|\sum_{\boldsymbol v\in H}
		\left(\sum_i c_i v_i\right)^2\chi(\boldsymbol v)
		 \equiv-|H|\mod3 \quad \text{and} \quad
		\sum_{\boldsymbol v\in H}
		\left(\sum_i c_i v_i\right)^2\chi(\boldsymbol v)
		\equiv-1\mod3,
\]
	where the second congruence follows from $3\nmid|H|$.
	
	Suppose that $\chi$ is neither one of the characters
	$\boldsymbol v\mapsto(-1)^{ v_i}$ nor one of the characters
	$\boldsymbol v\mapsto(-1)^{ v_i+ v_j}$ for distinct
	$i, j \in [1,13]$.
	By Lemma~\ref{orth-relations}, we obtain that
	\[
	\begin{aligned}
		\sum_{\boldsymbol v\in H}\chi(\boldsymbol v)=0,  \quad
		\sum_{\boldsymbol v\in H}(-1)^{ v_i}\chi(\boldsymbol v)=0, \quad \text{and} \quad
		\sum_{\boldsymbol v\in H}(-1)^{ v_i+ v_j}
		\chi(\boldsymbol v)=0.
	\end{aligned}
	\]
	Expanding the square of the right-hand side of
	\eqref{eq:davenport-character-expansion}, we get
	\begin{align*}
		&\sum_{\boldsymbol v\in H}\chi(\boldsymbol v)
		\Bigg[
		(c_1+\ldots+c_{13})^2\cdot 1+
		\sum_{i=1}^{13}\Bigg(
		c_i^2(-1)^{2 v_i}
		-2c_i(c_1+\ldots+c_{13})(-1)^{ v_i}+\sum_{j\ne i}
		c_ic_j(-1)^{ v_i+ v_j}
		\Bigg)
		\Bigg]\\
		\equiv&\sum_{\boldsymbol v\in H}
		\left(\sum_{i=1}^{13}c_i v_i\right)^2\chi(\boldsymbol v)\mod3,
	\end{align*}
which	gives $-1\equiv0\mod3$, a contradiction.
	{Thus every $\chi\in\widehat H\setminus\{\chi_0\}$ is either one of the characters $\boldsymbol v\mapsto(-1)^{v_i}$ for $i\in[1,13]$, or one of the characters $\boldsymbol v\mapsto(-1)^{v_i+v_j}$ for distinct $i,j\in[1,13]$.}
	Consequently, $|H|-1\le13+\binom{13}{2}=91$, and
	because $|H|\ge64$, we obtain $64\le|H|\le92$.
	Since $|H|$ is a power of $2$, we get $|H|=64$. \qedhere \ ({\bf A3})
\end{proof}

If $(-1)^{ v_i}=(-1)^{ v_j}$ for every
$\boldsymbol v\in H$, then
\[
c_i(-1)^{ v_i}+c_j(-1)^{ v_j}
=(c_i+c_j)(-1)^{ v_i}.
\]
Apply this operation to
\eqref{eq:davenport-character-expansion}, combining also
terms satisfying $(-1)^{ v_i}=1$ for every
$\boldsymbol v\in H$ with the constant term.
After reducing the coefficients modulo $3$ and omitting
zero coefficients, we obtain
\[
\sum_{i=1}^{13}c_i v_i
\equiv\sum_{j=1}^t a_j\chi_j(\boldsymbol v)\mod3
\]
for every $\boldsymbol v\in H$, where $t\le14$,
$a_1,\ldots,a_t\in\{1,-1\}$, and
$\chi_1,\ldots,\chi_t\in\widehat H$ are pairwise distinct.
	
\smallskip
\noindent
\begin{enumerate}
\item[{\bf A4.}\,] $t=12$.
\end{enumerate}

\begin{proof}[Proof of {\bf A4}]
By \eqref{eq:davenport-square-values} and
	$|H|=64\equiv1\mod3$, we have
	\[
	\left(\sum_{j=1}^t a_j\chi_j(\boldsymbol v)\right)^2
	\equiv 1-\sum_{\theta\in\widehat H}\theta(\boldsymbol v) \equiv -\sum_{\theta\in\widehat H\setminus\{\chi_0\}}\theta(\boldsymbol v)
	\mod3.
	\]
	
	Summing over $H$ and applying
	Lemma~\ref{orth-relations}, we obtain
\[
64t=\sum_{\boldsymbol v\in H}\left(\sum_{j=1}^t a_j\chi_j(\boldsymbol v)\right)^2 \equiv-\sum_{\theta\in\widehat H\setminus\{1\}}\sum_{\boldsymbol v\in H}\theta(\boldsymbol v) =0\mod3,
\]
whence $t\equiv0\mod3$.
For each $\chi\in\widehat H\setminus\{\chi_0\}$, multiplication
by $\chi(\boldsymbol v)$ and summation over $H$ similarly give
\[
\sum_{\boldsymbol v\in H}\chi(\boldsymbol v)\left(\sum_{j=1}^t a_j\chi_j(\boldsymbol v)\right)^2=128\sum_{\substack{1\le j<k\le t\\\chi_j\chi_k=\chi}}a_ja_k\equiv
-\sum_{\substack{1\le j<k\le t\\\chi_j\chi_k=\chi}}a_ja_k\mod3
\]
and
\[
-\sum_{\theta\in\widehat H\setminus\{1\}}\sum_{\boldsymbol v\in H}\theta(\boldsymbol v)\chi(\boldsymbol v)=-64\equiv-1\mod3.
\]
Since the two resulting sums are congruent modulo $3$,
it follows that
\begin{equation}\label{eq:davenport-character-pairs}
	\sum_{\substack{1\le j<k\le t\\\chi_j\chi_k=\chi}}a_ja_k\equiv1\mod3.
\end{equation}
We define
\begin{equation} \label{m-chi}
m_\chi=\bigl|\{(j,k)\in[1,t]^2:j<k,\ \chi_j\chi_k=\chi\}\bigr| \quad \text{for $\chi\in\widehat H\setminus\{\chi_0\}$}
\end{equation}
and we set $m_{\chi_0}=0$.
	By \eqref{eq:davenport-character-pairs}, we have
	$m_\chi\ge1$, and therefore
	\[
	\binom t2=\bigl|\{(j,k)\in[1,t]^2:j<k\}\bigr|
	=\sum_{\chi\in\widehat H\setminus\{1\}}m_\chi\ge63.
	\]
	Together with $t \in [12, 14]$ and $t\equiv0\mod3$, this yields $t=12$. \qedhere\ ({\bf A4})
\end{proof}

\smallskip
\noindent
\begin{enumerate}
\item[{\bf A5.}\,] There are exactly three nontrivial characters satisfying
	$m_\chi=2$, while every other nontrivial character satisfies $m_\chi=1$ (with $m_{\chi}$ as in \eqref{m-chi}).
\end{enumerate}

\begin{proof}[Proof of {\bf A5}]
By {\bf A4}, we infer that
	\[
	\sum_{\chi\in\widehat H\setminus\{\chi_0\}}(m_\chi-1)
	=\binom{12}{2}-63=3.
	\]
{Since $m_\chi\ge1$ for every $\chi\ne\chi_0$, the preceding equality implies that $m_\chi\ge2$ for some $\chi\in\widehat H\setminus\{\chi_0\}$.} Thus there are distinct index pairs
	$(j,k)$ and $(\ell,m)$, with $j<k$ and $\ell<m$, such that
	\[
	\chi_j\chi_k=\chi_\ell\chi_m=\chi.
	\]
	These pairs have no index in common. Indeed,
	$\chi_j\chi_k=\chi_j\chi_\ell$ implies
	$\chi_k=\chi_\ell$, and hence $k=\ell$, since
	$\chi_1,\ldots,\chi_t$ are distinct.
	After renumbering, we therefore have
	$\chi_1\chi_2=\chi_3\chi_4$,
	where $\chi_1,\chi_2,\chi_3,\chi_4$ are distinct.
	Since every character has square $1$, multiplying both sides of
	$\chi_1\chi_2=\chi_3\chi_4$ by $\chi_2\chi_3$ gives
	$\chi_1\chi_3=\chi_2\chi_4$.
	Similarly, multiplying both sides by $\chi_2\chi_4$ gives
	$\chi_1\chi_4=\chi_2\chi_3$.
	These three products are nontrivial and pairwise distinct.
	Thus
	\[
	m_{\chi_1\chi_2}\ge2,\qquad
	m_{\chi_1\chi_3}\ge2,\qquad
	m_{\chi_1\chi_4}\ge2.
	\]
	Since these three characters are distinct, all terms in
	$\sum_{\chi\in\widehat H\setminus\{\chi_0\}}(m_\chi-1)$
	are nonnegative, and their sum equals $3$, we obtain
	\[
	m_{\chi_1\chi_2}=m_{\chi_1\chi_3}=m_{\chi_1\chi_4}=2,
	\]
	and $m_\chi=1$ for every
	$\chi\in\widehat H\setminus
	\{1,\chi_1\chi_2,\chi_1\chi_3,\chi_1\chi_4\}$.
	\qedhere\ ({\bf A5})
\end{proof}
	
For $\chi=\chi_1\chi_2$, \eqref{eq:davenport-character-pairs}
	gives $a_1a_2+a_3a_4\equiv1\mod3$,
	and hence $a_1a_2=a_3a_4=-1$.
	Applying the same argument to $\chi_1\chi_3$ and
	$\chi_1\chi_4$, we obtain $a_1a_3=a_2a_4=-1$ and
	$a_1a_4=a_2a_3=-1$, respectively.
	In particular, $a_1a_2=a_1a_3=a_2a_3=-1$.
	Multiplying these three equalities gives
	$1=a_1^2a_2^2a_3^2=(-1)^3=-1$, a contradiction. \qedhere \ ({\bf A2})
\end{proof}
\qedhere \ ({Proposition \ref{lem:davenport}})
\end{proof}

\section{Some new techniques for studying $\nu (G)$} \label{sec:3}

Let $G$ be a finite abelian group.
By a {\it proper coset} (with respect to a subgroup $H \subseteq G$) we mean a coset of the form $\alpha + H$ with $\alpha \in G \setminus H$. An index $p$-coset is a coset $\alpha + H$ with $(G \colon H)=p$.

\begin{lemma}\label{lem:smallest-coset}
Let $G$ be an abelian group, $A\subseteq G$ a nonempty subset, $H \subseteq G$ a subgroup, and let $\alpha \in G \setminus H$. If $A \subseteq \alpha + H$,  then $A \subseteq \beta + \langle A-A \rangle \subseteq \alpha + H$ for every $\beta \in A$.
\end{lemma}

\begin{proof}
Let $\beta \in A$. Then $a - \beta \in A - A$ for every $a \in A$, whence $A \subseteq \beta + \langle A - A \rangle$. If $A \subseteq \alpha + H$, then $A - A \subseteq H$ and $\beta \in \alpha + H$, which implies that
\[
A \subseteq \beta + \langle A - A \rangle \subseteq \beta + H \subseteq \alpha + H \,. \qedhere
\]
\end{proof}

Our first result provides a condition implying that the equality  $\nu (G)=\mathsf d (G)$ extends from a group $G$ to all groups containing $G$ as a direct summand.

\begin{proposition}\label{prop:group-extension}
Let $G$ be a nontrivial finite abelian group with $G=G_1\oplus G_2$, where $G_1$ and $G_2$ are nontrivial subgroups of $G$.  If $\nu(G_1)=\dd(G_1)$ and
$\dd(G)=\dd(G_1)+\dd(G_2)$, then $\nu(G)=\dd(G)$.
\end{proposition}

\begin{proof}
Since $\nu(G_1)=\dd(G_1)$, there exists some zero-sum free sequence
$T \in \mathcal F (G_1)$ with $|T|=\dd(G_1)-1$ such that
$G_1\setminus\Sigma^*(T)$ is not contained in a proper coset.
We choose some zero-sum free  $U \in \mathcal F (G_2)$ of length
$\dd(G_2)$, and note that $\Sigma^*(U)=G_2$. {Then $S=TU$ is zero-sum free, and}
\[
|S|=|T|+|U|  =\dd(G_1)-1+\dd(G_2) =\dd(G)-1 \quad \text{and} \quad
\Sigma^*(S)=\Sigma^*(T)+\Sigma^*(U) =\Sigma^*(T)+G_2 \,.
\]
Since $G=G_1\oplus G_2$ and $\Sigma^*(T)\subseteq G_1$, we obtain
$G\setminus\Sigma^*(S)
 =\bigl(G_1\setminus\Sigma^*(T)\bigr)+G_2,$
and hence
$$\bigl(G\setminus\Sigma^*(S)\bigr)\cap G_1
 =G_1\setminus\Sigma^*(T).$$
{Assume to the contrary that}
$G\setminus\Sigma^*(S)\subseteq \alpha +H$
for some subgroup $H \subseteq G$ and some $\alpha \in G\setminus H$.
Intersecting with $G_1$ gives
$$G_1\setminus\Sigma^*(T)\subseteq(\alpha  +H)\cap G_1.$$
We choose $\alpha' \in G_1\setminus\Sigma^*(T)$. Then $\alpha' \in(\alpha+H)\cap G_1$, and therefore
$$(\alpha +H)\cap G_1=\alpha'+(H\cap G_1).$$
Since $\alpha \notin H$ and $\alpha' \in \alpha+H$, we have $\alpha' \notin H$.
Thus $H\cap G_1$ is a proper subgroup of $G_1$ and
$\alpha' \notin H\cap G_1$. It follows that
$$G_1\setminus\Sigma^*(T)\subseteq \alpha' +(H\cap G_1) \,,$$
a contradiction to  the choice of $T$.
Thus,  $S \in \mathcal F (G)$ is zero-sum free
of length $\dd(G)-1$ and $G^{\bullet} \setminus \Sigma (S)$ is not contained in a proper coset, whence $\nu (G) = \mathsf d (G)$ by {\eqref{def-nu(G)}}.
\end{proof}


\begin{lemma} \label{lem:3.3}
Let $\varphi \colon G \to G_0$ be a group epimorphism and let $k \in \N$. If  $S \in \mathcal F(G)$ and  $|S|\ge (k-1) \, \exp
(G_0) + \eta(G_0)$, then $S$ admits a product
decomposition \ $S = S_1 \cdot \ldots \cdot S_k S'$, where $S_1,
\ldots, S_k, S' \in \mathcal F(G)$ and, for every $i \in [1,k]$, \
$\varphi (S_i)$ has sum zero and length  \ $|S_i| \in [1, \exp
(G_0)]$.
\end{lemma}

\begin{proof}
See \cite[Lemma 5.7.10]{Ge-HK06a}.
\end{proof}

\begin{lemma}\label{lem:subgroup-quotient-bounds}
Let $G$ be a {nontrivial} finite abelian group and let $H {\subsetneq} G$ be a subgroup. Then
\begin{align*}
\dd(G)&\le\dd(H) \exp (G/H)+\max\{\dd(G/H),\eta(G/H)- \exp (G/H)-1\},\\
\nu(G)&\le\dd(H) \exp (G/H) +\max\{\nu(G/H),\eta(G/H)- \exp (G/H) \}.
\end{align*}
Moreover, for every prime divisor $p$ of $|G/H|$,
$$\nu_p(G)\le\dd(H) \exp (G/H) +\max\{\nu_p(G/H),\eta(G/H)- \exp (G/H) \}.$$
\end{lemma}

\begin{proof}
The bound for $\dd(G)$ follows from \cite[Proposition~5.7.11]{Ge-HK06a}, whence it suffices to
prove the bounds for $\nu (G)$ and for $\nu_p (G)$.
{If $H=\{0\}$}, then all claims are {clear.} Suppose that $\{0\} \ne H \subsetneq G$,
and let $\varphi:G\to G/H$ be the canonical epimorphism.
Let $S \in \mathcal F (G)$ be zero-sum free with
\[
|S|\ge \mathsf d (H) \exp (G/H) +\max\{\nu(G/H),\eta(G/H)- \exp (G/H)\} \,.
\]
{To prove the asserted bound for $\nu(G)$, it suffices by definition to show that $G\setminus\Sigma^*(S)$ is contained in a proper coset.}

By Lemma \ref{lem:3.3}, $S$ has a product decomposition $S=S_1 \cdot \ldots \cdot S_{\mathsf d (H)}S'$, where $S_1, \ldots, S_{\mathsf d (H)}, S' \in \mathcal F (G)$ with
 $|S_i|\le \exp (G/H)$ and
$\sigma(\varphi(S_i))=0$.
The sequence $B=\prod_{i=1}^{\mathsf d (H)}\sigma(S_i) \in \mathcal F (H)$ is zero-sum free  and $|B|=\dd(H)$,
whence $\Sigma^*(B)=H$.

{If there were} a nontrivial sequence $S''$ with $S'' \mid S'$ and $\sigma (S'') \in H$, then $\sigma (S'')B$ would not be zero-sum free, whence $S$ would not be zero-sum free. Thus,
the sequence $\varphi(S')$ is zero-sum free and by assumption we have $|S'|\ge\nu(G/H)$.
By definition of $\nu (G/H)$, there are a subgroup $H_1\subsetneq G/H$ and some
$\alpha\in G/H\setminus H_1$ such that
$G/H \setminus \Sigma^*(\varphi(S'))\subseteq \alpha+H_1$.
Since the complete blocks represent every element of $H$ and
are disjoint from $S'$, we have
$$\varphi^{-1}\bigl(\Sigma^*(\varphi(S'))\bigr)
  =H+\Sigma^*(S')\subseteq\Sigma^*(S).$$
Consequently, $G\setminus\Sigma^*(S)\subseteq\varphi^{-1}(\alpha+H_1)$,
a proper coset of $\varphi^{-1}(H_1)$.

To prove the bound for $\nu_p (G)$, we
replace $\nu(G/H)$ by $\nu_p(G/H)$ and proceed with   the same argument. {Since $(G/H:H_1)=p$, we have $(G:\varphi^{-1}(H_1))=p$, and the claim follows.}
\end{proof}

\begin{proposition}\label{prop:quotient-lifting}
Let $G_0$ be a nontrivial finite abelian group, let $k, m\in \N$
with $\exp(G_0)\mid m$, and let $G_k=G_0\oplus C_{km}$.
Suppose that
$$\dd(G_1)=\dd(G_0)+m-1,\quad
 \nu(G_1)=\dd(G_1)-1,\quad \text{and} \quad
 \eta(G_1)\le\dd(G_0)+2m-2.$$
Then
$$\dd(G_k)=\dd(G_0)+km-1 \qquad \text{and} \qquad \nu(G_k)=\dd(G_k)-1.$$
If, in addition, $\nu_p(G_1)=\dd(G_1)-1$ for a prime divisor
$p$ of $|G_1|$, then $\nu_p(G_k)=\dd(G_k)-1$.
\end{proposition}

\begin{proof}
Let $\varphi:G_k \to G_1$
be the identity on $G_0$ and reduction modulo $m$ on the cyclic factor.
Then $H=\ker (\varphi) \cong C_k$, $\dd(H)=k-1$, $G_k/H \cong G_1$, and $\exp(G_1)=m$.
Since $\eta(G_1)\le\dd(G_1)+m-1$ by assumption,
Lemma~\ref{lem:subgroup-quotient-bounds} and Inequality \eqref{basic-inequ-4} imply that
\[
km-1 + \mathsf d (G_0) \le \dd(G_k)\le(k-1)m+\dd(G_1)=\dd(G_0)+km-1 \,.
\]
Furthermore, Lemma~\ref{lem:subgroup-quotient-bounds} implies that
$$\nu(G_k)\le(k-1)m+\dd(G_1)-1=\dd(G_k)-1,$$
and Inequality \eqref{basic-inequ-2} shows that  equality holds.
Under the additional assumption on $\nu_p(G_1)$, the last bound
in the lemma gives $\nu_p(G_k)\le\dd(G_k)-1$, and again we obtain equality.
\end{proof}

\begin{theorem} \label{thm:large}
Let $G_0$ be a  finite abelian $p$-group and let $m$ be a
power of $p$ such that $\dd(G_0)\le m-2$. Then
\[
\nu(G_0 \oplus C_{km})=\nu_p(G_0 \oplus C_{km})=\dd(G_0 \oplus C_{km})-1 \quad \text{for every $k \in \N$} \,.
\]
\end{theorem}

\begin{proof}
Let $k \in \N$ and let $G_k = G_0 \oplus C_{km}$.
 Lemma \ref{2.2} implies that $\dd(G_k)=\dd^*(G_k)=\dd(G_0)+km-1$. If $|G_0|=1$, then $G_k$ is cyclic whence the claim holds.
If $k=1$, then $G_k$ is a $p$-group, whence the claim holds by \cite[Theorem~3.5]{Ge-Ya27a}. Suppose  that $|G_0| >1$ and $k > 1$.

Let  $G_1=G_0\oplus C_m$.
Since $\exp(G_0)-1\le \mathsf d (G_0) \le m-1$ and both exponents are
powers of $p$, $\exp(G_0)\mid m$ and $\exp (G_1)=m$.
Since $\mathsf d (G_1)=\mathsf d (G_0) +m-1\le2m-2$, Lemma \ref{2.3} implies that
${\eta(G_1)= 2 \mathsf D (G_1) - \exp (G_1)= 2\mathsf d (G_0) +m\le \mathsf d (G_0) +2m-2}$.

Let $S \in \mathcal F (G_k)$ be  zero-sum free  with $|S| = \mathsf d (G_k)-1$. We need to show that $G_k \setminus \Sigma^* (S)$ is contained in a proper coset of index $p$.
Let
$\varphi:G_k\to G_1$ be the identity on $G_0$ and reduction modulo $m$ on the cyclic factor, whence  $H = \ker (\varphi) \cong C_k$.
Since
\[
{|S| = \mathsf d (G_0)+km-2 = (k-2)m+\mathsf d (G_0)+2m-2 \ge (k-2)m + \eta(G_1)} \,,
\]
Lemma \ref{lem:3.3} implies that $S$ admits a product decomposition $S = S_1 \cdot \ldots \cdot S_{k-1}S'$, where $S_1, \ldots, S_{k-1}, S' \in \mathcal F (G_k)$ and, for every $i \in [1,k-1]$, $\varphi (S_i)$ has sum zero and length $|S_i| \in [1, \exp (G_1)]$.
Then $U= \sigma (S_1) \cdot \ldots \cdot \sigma (S_{k-1}) \in \mathcal F (H)$ is zero-sum free of length
$|U|=\dd(H)$, whence $\Sigma^*(U)=H$.
{If $S'$ had a} {nontrivial} subsequence $T$ such that $\varphi (T)$ has sum zero, then $U \sigma (T)$ and hence $S$ would have {a nontrivial zero-sum subsequence}. Thus $\varphi(S')$ is zero-sum free.
Since $G_1$ is a $p$-group and $\nu_p (G_1) = \mathsf d (G_1)-1$, we obtain that
\[
|S'| = |S| - |S_1 \cdot \ldots \cdot S_{k-1}| \ge \mathsf d (G_k)-1 - (k-1)m  = \mathsf d (G_0)+m-2=\mathsf d (G_1)-1 = \nu_p(G_1)
\]
and
$\varphi^{-1}(\Sigma^*(\varphi(S')))=H+\Sigma^*(S')\subseteq\Sigma^*(S)$.
Since $G_1 \setminus \Sigma^* \big( \varphi(S') \big) \subseteq \alpha_1 + H_1$, where  $H_1 \subseteq G_1$ is a subgroup of index $p$ and $\alpha_1 \in G_1 \setminus H_1$, the above
inclusion shows that  $G_k \setminus \Sigma^* (S)$ lies in the inverse image
of ${\alpha_1 + H_1}$, whence it lies in  a proper coset of an index-$p$ subgroup of $G_k$.
\end{proof}

\section{Proof of Theorem \ref{thm:main1}} \label{sec:4}

Let $G$ be a finite abelian group. {For cyclic groups and groups of rank two, the precise value of $\eta(G)$ is known ({see Equation~\eqref{basic-inequ-3}}), as is the structure of sequences of length $\eta(G)-1$ with no {nontrivial} zero-sum subsequence of length at most $\exp(G)$ (\cite{Sc12a}).} Much less is known for higher ranks. For some special classes of groups with rank three {the direct and inverse problems for} $\eta (G)$ are solved (\cite{Gi-Sc19a,Gi-Sc20a}), but the precise value of $\eta (C_p^3)$, where $p$ is a prime, is not known in general. The next lemma settles the case for the group $C_2^3 \oplus C_4$.

\begin{lemma}\label{lem:short-free-nine}
Let $G_0=C_2^3\oplus C_{2n}$ with ${n\in\N_{\ge2}}${.}
\begin{enumerate}
\item  ${\eta(G_0)=2n+6}$.

\item Let $n=2$ and let $S \in \mathcal F (G_0)$ with $|S|= 9$. Then the following are equivalent.
      \begin{enumerate}
      \item[(a)]  $S$ has  no nontrivial zero-sum subsequence of length at most four.

      \item[(b)] $S=e^3a_1a_2a_3(e+b_1)(e+b_2)(e+b_3)$, where $G_0=H_1\oplus\langle e\rangle$,  $\ord(e)=4$ and $(a_i)_{i=1}^3$, $(b_i)_{i=1}^3$ are
           bases of $H_1\cong C_2^3$.
      \end{enumerate}
\end{enumerate}
\end{lemma}

\begin{proof}
1.  See  {\cite[Theorem~1.2(2)]{FanZhong2016}}.

2. (b) $\Longrightarrow$ (a) {We verify that no nontrivial subsequence $T\mid S$ of length at most four has sum zero. If the cyclic coordinate of $\sigma(T)$ is zero, then $T$ contains either zero or four terms from $e^3(e+b_1)(e+b_2)(e+b_3)$. In the first case, $T\mid a_1a_2a_3$, so $\sigma(T)\ne0$ by the independence of the $a_i$. In the second case, $|T|\le4$ implies that $T$ contains none of the $a_i$. Since $\mathsf v_e(S)=3$, at least one term $e+b_i$ occurs. As $4e=0$, $\sigma(T)$ is a nonempty sum of distinct $b_i$ and is therefore nonzero.}

(a) $\Longrightarrow$ (b)
Note that  $G_0[2]\cong C_2^4$ and set $2G_0=\{0,z\}$.
By \cite[Corollary~1.4]{Ga-Ho-Sc-Th07}, we have $\mathsf v_g (S) \in \{1,3\}$ for every $g \in \supp (S)$. If $\mathsf v_g (S)=3$, then $\ord (g)=4$. {If there were} distinct $g_1, g_2 \in \supp (S)$ with $\mathsf v_{g_1}(S)=\mathsf v_{g_2} (S)=3$, then $g_1^2g_2^2$ would have sum zero because $2G_0 = \{0,z\}$. {Thus at most one element of $\supp(S)$ has multiplicity three.}
Let $\supp (S) = \{a_1, \ldots, a_s\} \sqcup \{h_1, \ldots, h_t\}$, where $\ord ({a_i})=2$ for all $i \in [1,s]$ and $\ord (h_j)=4$ for all $j \in [1,t]$. {We set $U=a_1\cdot\ldots\cdot a_s$ and $V=h_1\cdot\ldots\cdot h_t$, and consider the four sets}
\[
\{0\},\ \supp(U),\ \Sigma_2(U), \quad \text{and} \quad \Sigma_2(V) \,.
\]
These are four disjoint subsets of $G_0[2]$ of sizes $1,s,\binom s2$, and $\binom t2$, respectively; otherwise, we would obtain a nontrivial zero-sum subsequence of $S$ of length at most four.

If $S$ were square-free, then $s+t=9$ would give
$|G_0[2]|=16\ge1+s+\binom s2+\binom t2=(s-4)^2+21$, a contradiction.
Thus there is a unique element $e \in \supp (S)$ with $e^3\mid S$, whence
\begin{equation}\label{equation:s+t=7}
 |Se^{-2}|=|UV|=s+t=7.
\end{equation}
Moreover, we assert that
\begin{equation}\label{equation:znotinSet}
 z \notin \{0\}\sqcup \supp(U) \sqcup \Sigma_2(U) \sqcup \Sigma_2(V) \,;
\end{equation}
otherwise, a representation of $z$ using at most two terms could
be combined with $e^2$ to give a zero sum of length at most four.
Consequently,
\begin{equation}\label{equation:s=2,3,4}
|G_0[2]|= 16\ge2+s+\binom s2+\binom t2=(s-3)^2+14,
 \quad \text{whence} \quad s \in [2,4] \,.
\end{equation}
For $(s,t)=(2,5)$ or $(4,3)$, the union $X$ of these four sets
and $\{z\}$ has size fifteen. Its sum
$\sum_{x\in X}x=z+s\sigma(U)+(t-1)\sigma(V)$
equals $z$ or $0$, respectively, since $2\sigma(V)=tz$.
As $\sum_{x\in G_0[2]}x=0$, the unique omitted point would therefore
lie in $X$, a contradiction.

{Therefore, $(s,t)=(3,4)$}, say
$U=a_1a_2a_3$ and $V=e\prod_{i=1}^3(e+b_i)$, with $b_i\in G_0[2]$.
Obviously, $a_1, a_2, a_3$ are independent. The elements $b_1, b_2, b_3$ are distinct and nonzero,
and their sum is nonzero, since otherwise $V$ would have sum zero.
Thus, $b_1, b_2, b_3$ are independent.
{We show that the subgroups $H_1=\langle a_1,a_2,a_3\rangle$ and $H_2=\langle b_1,b_2,b_3\rangle$ coincide.}
We have
\[
H_1\setminus\{\sigma(U)\} = \{0\}\cup\supp(U)\cup\Sigma_2(U) \quad
\text{and} \quad (z+H_2)\setminus\{z+b_1+b_2+b_3\} = \{z\}\cup\Sigma_2(V) \,.
\]
If $H_1\ne H_2$, then $H_1+H_2=G_0[2]$ and $|H_1\cap H_2|=4$.
{Since $z\in H_1+H_2$, the intersection $H_1\cap(z+H_2)$ is nonempty. If $x$ belongs to this intersection, then $H_1\cap(z+H_2)=x+(H_1\cap H_2)$.} The union of $H_1$ and $z+H_2$ consequently has
size twelve, although it contains the fourteen elements above,
a contradiction. Thus $H_1=H_2$.
Furthermore, we observe that  $z\notin H_1$, because fourteen elements cannot
lie in $H_1$. Since $2e=z$, we finally obtain that $G_0=H_1\oplus\langle e\rangle$.
\end{proof}

\begin{lemma}\label{lem:algebraic-lifting}
Let $G_0 = H_1 \oplus \langle e \rangle$ with $\ord (e)=4$ and $H_1 \cong C_2^3$.
Let $G$  be a finite abelian group and  let $\varphi:G\to G_0$ be an epimorphism. Let $S \in \mathcal F (G)$ be such that   $\varphi(S)$ {has the form given in}  Lemma~\ref{lem:short-free-nine}, {with the notation used there}.
Suppose that {any two nontrivial subsequences} $T_1$ and $T_2$ of $S$ with
$\sigma(\varphi(T_1))=\sigma(\varphi(T_2))=0$ satisfy $\sigma (T_1)=\sigma (T_2) \in \ker (\varphi) \setminus\{0\}$.
Then for every  $h \in H_1\oplus\langle2e\rangle$  there are distinct $u, v \in \Sigma^*(S)$ with $\varphi (u)=\varphi (v)=h$.
\end{lemma}

\begin{proof}
{We first reduce the assertion to $h\in H_1$. For every subsequence $T\mid S$, its complement $ST^{-1}$ has sum $\sigma(S)-\sigma(T)$. Thus the map $u\mapsto\sigma(S)-u$ is a bijection of $\Sigma^*(S)$ that preserves distinctness. Since $\sigma(\varphi(S))\in H_1+2e$, this map interchanges the fibres over $H_1$ and $H_1+2e$. It therefore suffices to prove the assertion for $h\in H_1$.}

Let $h \in H_1$ be given. Since $(a_1,a_2,a_3)$ is a basis of $H_1$, there are $x_1,x_2,x_3 \in G$ such that $x_1x_2x_3 \mid S$, $\varphi (x_i)=a_i$ for $i \in [1,3]$ and there is a subsequence $U$ of $x_1x_2x_3$ such that $\sigma ( \varphi (U) ) = h$.
Assume to the contrary that every subsequence $T$ of $S$ with
$\sigma ( \varphi ({T}) ) = h$ has sum $\sigma(U)$.

We next show that every element of $H_1$ occurs as a projected subsequence sum of $S(x_1x_2x_3)^{-1}$. If $0\ne h\in H_1$, then $h$ is a sum of $\ell$ distinct $b_j$, where $\ell\in[1,3]$. From $S(x_1x_2x_3)^{-1}$, select the corresponding $\ell$ terms with projections $e+b_j$ and $4-\ell$ of the three terms with projection $e$, counting terms with their multiplicities. The resulting subsequence has projected sum $4e+h=h$. For $h=0$, the empty subsequence has projected sum zero.

{Note that $\Sigma \big( e^3(e+b_1)(e+b_2)(e+b_3) \big) = G_0^{\bullet}$. Thus,}
for $i \in [1,2]$, there is a subsequence
$V_i$ of $S (x_1x_2x_3)^{-1}$ with $\sigma(\varphi(V_i))=a_i$. Since $\sigma ( \varphi (x_i)) = \varphi (x_i)=a_i$, it follows that $\sigma ( \varphi (x_iV_i)) = 0$. {The hypothesis therefore gives}
$\sigma(x_1V_1)= \sigma (x_2V_2) =: g \in \ker ( \varphi) \setminus \{0\}$.

If $x_i\nmid U$, then $Ux_iV_i$ has projected sum $\varphi ( \sigma (Ux_iV_i)) = h$ and
sum $\sigma (Ux_iV_i)= \sigma(U)+g$, a contradiction to $g\ne0$. {Thus $x_i\mid U$.}
The subsequence $(Ux_i^{-1})V_i$ also has projected sum $\sigma ( \varphi (Ux_i^{-1}V_i) ) = h$,
so $\sigma(V_i)=x_i$ and consequently $2x_i=g$.

Since $a_1+a_2\ne0$, there is a subsequence $V$ of
$S(x_1x_2x_3)^{-1}$ with $\sigma(\varphi(V))=a_1+a_2$.
The subsequences $(U (x_1x_2)^{-1})V$ and $U$ have the same projected sum, so our assumption gives
$\sigma(V)=x_1+x_2$. The hypothesis applied to $x_1x_2V$
therefore gives $g=2x_1+2x_2=2g$, a contradiction.
\end{proof}

\begin{proof}[Proof of Theorem~\ref{thm:main1}]
Let $G=C_2^3\oplus C_{2n}$. If $n \in \{1,2\}$, then the claim follows from the result on $2$-groups (\cite[Theorem~3.5]{Ge-Ya27a}).

1. Suppose that  $n\ge3$ is odd.  By
Inequality~\ref{basic-inequ-2}, it suffices to show that  $\nu(G)>\dd(G)-1$. Thus, we need to find a zero-sum free sequence $S \in \mathcal F (G)$ of length $|S|=\mathsf d (G)-1$ such that $G^{\bullet} \setminus \Sigma (S)$ is not contained in a proper coset.

To do so, we fix a basis
$(e_1,e_2,e_3,e)$ of $G$ with $\ord(e_i)=2$ for $i \in [1,3]$ and $\ord(e)=2n$, and
set $a=(n+1)/2$. Then we define
\begin{equation}\label{eq:family}
 S=e^{2n-3}(e_1+ae)(e_2+ae)(e_3+ae)  (e_1+e_2+e_3+ae)\,{.}
\end{equation}

{We verify that $S$ is zero-sum free. By independence of $e_1,e_2,e_3$, a subsequence with zero $C_2^3$-coordinate contains either none or all four of the last terms. For a nontrivial such subsequence, the cyclic coefficient belongs to $[1,2n-3]$ in the first case and to $4a+[0,2n-3]=2n+[2,2n-1]$ in the second. Neither interval contains a multiple of $2n$.} By Lemma~\ref{2.1}.2, we have
$|S|=2n+1=\dd(G)-1$ .

Let $A=\{e_1+e_2,e_1+e_3,e_2+e_3\}$.
Any subsequence $T$ of $S$ with sum  $\sigma (T) \in x+\langle e\rangle$ for some $x\in A$ {contains}
exactly two of the last four terms, so its cyclic coefficient lies in
$n+1+[0,2n-3]$. Therefore
\begin{equation}\label{equation:missing-sums}
 \{x+(n-1)e,x+ne:x\in A\}
 \subseteq G\setminus\Sigma^*(S).
\end{equation}
Assume to the contrary that $G\setminus\Sigma^*(S) \subseteq  \alpha+H$ for some subgroup $H \subseteq G$ and some  $\alpha\in G \setminus H$.
Taking differences, we obtain $e\in H$ and $A-A\subseteq H$.
Since $A-A=A\cup\{0\}$, it follows that  $\{x+(n-1)e,x+ne:x\in A\} \subseteq  H$, a contradiction.

2. Suppose that  $n=2m$ with $m \in \N_{\ge 2}$. Then
$\dd(G)=4m+2$ by Lemma~\ref{2.1}.
Let $S \in \mathcal F (G)$ {be a zero-sum free sequence} of length $|S|= \mathsf d (G)-1${. We need} to show that $G^{\bullet} \setminus \Sigma (S)$ is contained in a proper coset of index two.

Let $H = 4G$ and let $\varphi:G\to G/H$ be the natural epimorphism. Then $H$ is cyclic of order $m$ and $\exp (G/H)=4$.
Let $k \in {\N_0}$ be maximal such that $S$ has a product decomposition $S=S_1 \cdot \ldots \cdot S_kS'$, where $S_1, \ldots, S_k,S' \in \mathcal F (G)$ and, for every $i \in [1,k]$, $\varphi (S_i)$ has sum zero and $|S_i| \in [1, 4]$.
Then $U = \prod_{i=1}^k\sigma(S_i) \in \mathcal F (H)$ is zero-sum free, whence $k \le m-1$. {For $k=0$, the product of the blocks is empty and $U=1$.}
Maximality and $\eta(G/H)=10$ (by Lemma \ref{lem:short-free-nine}.1) imply $|S'|\le9$.
Thus $k \in [m-2, m-1]$, $|S'|\ge5$, and
$|\Sigma^*(U)|\ge k +1\ge m-1${. Moreover,}
\begin{equation}\label{eq:short-proof-block-sums}
 \Sigma^*(U)+\Sigma^*(S')\subseteq\Sigma^*(S).
\end{equation}
We distinguish two cases.

\noindent
CASE 1:  $\Sigma^*(U)=H$.

Assume to the contrary that $S'$ has a nontrivial subsequence $T'$ with $\sigma (T') \in H$. Then $\sigma (T')U \in \mathcal F (H)$ has a zero-sum subsequence, whence $S$ has a nontrivial zero-sum subsequence, a contradiction.
Thus $\varphi(S') \in \mathcal F (G/H)$ is zero-sum free.
As $\nu_2(G/H)=\mathsf d (G/H)-1=5$ by \cite[Theorem~3.5]{Ge-Ya27a}, $(G/H) \setminus \Sigma^* \big( \varphi (S') \big)$ lies in a proper index-two coset.
By \eqref{eq:short-proof-block-sums},  $G \setminus \Sigma^* (S)$
lies in the inverse image of this coset, which
is a proper coset of an index-two subgroup of $G$.

\noindent
CASE 2: $\Sigma^*(U)\ne H$.

This implies that
$k=m-2$ and $\mathsf d (G)-1=|S| \le \sum_{i=1}^k |S_i| + |S'| \le 4(m-2) + |S'|$, whence $|S'|=9$.

We show that $\Sigma^*(U)$ omits exactly one element of $H$. The bound $|\Sigma^*(U)|\ge m-1$, together with $\Sigma^*(U)\ne H$, gives $|\Sigma^*(U)|=m-1$. Since $0\in\Sigma^*(U)$, we obtain $\Sigma^*(U)=H\setminus\{h\}$ for some $0\ne h\in H$.
Every nontrivial subsequence $V$ of $S'$ with $\sigma(\varphi(V))=0$ has sum $-h$,
since otherwise  $- \sigma (V)$ would lie in $\Sigma^*(U)$, and we would obtain a zero-sum subsequence of  $S$.

{We now verify the hypotheses of Lemma~\ref{lem:algebraic-lifting} for $S'$ and $\varphi$. By maximality of $k$, $\varphi(S')$ has no nontrivial zero-sum subsequence of length at most four. Since $|S'|=9$, Lemma~\ref{lem:short-free-nine}.2 supplies the required structure of $\varphi(S')$. The equality of the nonzero sums established above supplies the remaining hypothesis. Hence Lemma~\ref{lem:algebraic-lifting} gives two distinct sums $u,v\in\Sigma^*(S')$ over every $h\in H_2=H_1\oplus\langle2e\rangle$.}
The translates $u+\Sigma^*(U)$ and $v+\Sigma^*(U)$ omit different points of that fibre,
so together cover it. Hence $\varphi^{-1}(H_2)\subseteq\Sigma^*(S)$,
and all missing sums lie in the nonzero index-two coset
$\varphi^{-1}(e+H_2)$.
\end{proof}

\section{Proof of Theorem~\ref{thm:main2}} \label{sec:5}

\begin{proof}[Proof of Theorem~\ref{thm:main2}]
1. Let $G_r = C_2^r \oplus C_6$ with $r \in \{5,6\}$. Lemma \ref{lem:davenport} shows that $\mathsf d (G_r) = \mathsf d^* (G_r)+1$. By Inequality \ref{basic-inequ-2}, it suffices to show that $\nu (G_r) > \mathsf d (G_r)-1$. Thus, we need to find a zero-sum free sequence $S \in \mathcal F (G_r)$ of length $|S|=\mathsf d (G_r)-1$ such that $G_r \setminus \Sigma^* (S)$ is not contained in a proper coset. Let $(e_1, \ldots, e_r, e)$ be a basis of $G_r$ with $\ord (e_i)=2$ for $i \in [1,r]$ and $\ord (e)=6$.

(i) Let $r=5$  and let
\begin{equation}\label{eq:S}
 S=\prod_{i=1}^5(e_i+e)(e_i+4e).
\end{equation}
If $1 \ne T$ is a subsequence of $S$ with $\sigma (T) \in \langle e \rangle$, then there is some $\emptyset \ne I \subset [1,5]$ such that $T = \prod_{i \in I} (e_i+e)(e_i+4e)$, whence $\sigma (T) = |I|(-e) \ne 0$.
Hence $S$ is zero-sum free and has length $|S|=10=\dd(G_5)-1$.

We set $g_0=e_1+\ldots+e_5$ and observe that
\begin{align*}
 \Sigma^*(S)\cap(g_0+\langle e\rangle)
   &=g_0+\{2e,5e\},\\
 \Sigma^*(S)\cap(g_0-e_i+\langle e\rangle)
   &=g_0-e_i+\{0,e,3e,4e\}.
\end{align*}
Thus
\[
A:=  \{g_0,g_0+e\}\cup\{g_0-e_i+2e:  i \in [1,5] \} \subseteq G_5 \setminus\Sigma^*(S) \,,
\]
whence $G_5 = \langle e, -e_i+2e \colon i \in [1,5] \rangle \subseteq \langle A - A \rangle$.
Now Lemma~\ref{lem:smallest-coset} shows that $A$ ( and hence $G_5 \setminus \Sigma^* (S)$) cannot be contained in a proper coset of $G_5$.

(ii) We have $G_6=G_5\oplus\langle e_6\rangle$, and for every $g\in G_5$, we set
\begin{equation}\label{eq:lift}
 S_g=S(e_6+g).
\end{equation}
Since $S$ is zero-sum free,  $S_g$
is zero-sum free,  $|S_g|=11=\dd(G_6)-1$, and
$\Sigma^*(S_g)\cap G_5 =\Sigma^*(S)$.

{We show that $G_6\setminus\Sigma^*(S_g)$ is not contained in a proper coset. Since $\Sigma^*(S_g)\cap G_5=\Sigma^*(S)$, we have $A\subseteq G_6\setminus\Sigma^*(S_g)$. Suppose to the contrary that $G_6\setminus\Sigma^*(S_g)\subseteq\alpha+H$ for some subgroup $H\subsetneq G_6$ and some $\alpha\in G_6\setminus H$. Then $A-A\subseteq H$, so $G_5=\langle A-A\rangle\subseteq H$. Thus the nonempty set $A$ lies in both $H$ and $\alpha+H$, contradicting $\alpha\notin H$.}

2.  If $n \in \{1,2\}$, then ${G}$ is a $2$-group, so the assertion follows from
\eqref{basic-inequ-1} and from \cite[Theorem~3.5]{Ge-Ya27a}.
{Suppose that $4\mid n$. {By Lemma~\ref{2.1}.1, we have $\dd(G)=\dd^*(G)$, since $\mathsf v_2(n)\ge2$ and $r\le6\le2^{\mathsf v_2(n)+1}-1$.} We apply Theorem~\ref{thm:large} with $G_0=C_2^r$, $m=8$, and $k=n/4$. Its hypothesis holds because $\dd(G_0)=r\le6=8-2$, whence the claim follows.}
\end{proof}

\end{document}